\documentclass[11pt]{proc-l}
\usepackage{amsmath}
\usepackage{amssymb}
\usepackage{amsfonts}
\usepackage{graphicx}

 \newtheorem{thm}{Theorem}[section]
 
 \newtheorem{lem}[thm]{Lemma}
 \newtheorem{prop}[thm]{Proposition}
 \theoremstyle{definition}

 \theoremstyle{remark}

\numberwithin{equation}{section}
\numberwithin{figure}{section}

\newcommand{\CC}{{\mathbb C}}

\newcommand{\TT}{{\mathbb T}}

\newcommand{\NN}{{\mathbb N}}

\newcommand{\EE}{\mathbb{E}}
\newcommand{\PP}{\mathbb{P}}
\begin{document}
\bibliographystyle{alpha}

\title[Expected discrepancy for zeros of random polynomials]
{Expected discrepancy for zeros of random algebraic polynomials}

%----------Author 1
\author[Pritsker]
{Igor E. Pritsker}

\address{Department of Mathematics, Oklahoma State University, Stillwater, OK 74078, USA.}
\email{igor@math.okstate.edu}

%----------Author 2
\author[Sola]
{Alan A. Sola}

\address{Statistical Laboratory, University of Cambridge, Cambridge CB3 0WB,UK.}
\email{a.sola@statslab.cam.ac.uk}

%----------classification, keywords, date
\date{8/18/2012}
\thanks{Pritsker acknowledges support from NSA under grant H98230-12-1-0227; Sola acknowledges support from the EPSRC under grant EP/103372X/1.}

\dedicatory{Dedicated to Vladimir Andrievskii on his 60th birthday}

\begin{abstract}
We study asymptotic clustering of zeros of random polynomials, and show that the expected discrepancy of roots of a polynomial of degree $n$, with not necessarily independent coefficients, decays like $\sqrt{\log n/n}$. Our proofs rely on discrepancy results for deterministic polynomials, and order statistics of a random variable. We also consider the expected number of zeros lying in certain subsets of the plane, such as circles centered on the unit circumference, and polygons inscribed in the unit circumference.
\end{abstract}

\maketitle

\section{Introduction}
\subsection{Random polynomials and their zeros}
Let $\{C_k\}_{k=0}^{\infty}$ be a sequence of independent and identically distributed (iid) complex-valued random variables. In this paper, we study families of random polynomials
\begin{equation} \label{generalrandompoly}
P_n(z)=\sum_{k=0}^{n}C_k z^k
\end{equation}
and the geometry of their zeros; we let $\mathcal{Z}(P_n)=\{Z_1,\ldots,Z_n\}$ denote the set of complex zeros of such a polynomial of degree $n$. We use the notation $P_n(z)=\sum_{k=0}^n c_kz^k$ whenever we make statements that apply to any polynomial with $c_0,\ldots,c_n \in \CC$ and reserve capital letters for random coefficients. The zeros $\{Z_k\}_{k=1}^n$ of a random polynomial $P_n$ define a natural random measure on $\CC$, the {\it counting measure of roots} or {\it empirical measure}
\[\tau_n=\frac{1}{n}\sum_{k=1}^n\delta_{Z_k}.\]

Random polynomials of the form \eqref{generalrandompoly} have been studied by % Littlewood and Offord, Kac, Rice, Hammersley, Arnold, and
many authors, and it is known that, under mild conditions on the distribution of the coefficients of $P_n$, these empirical measures converge to $\mu_{\TT}$, the normalized arc-length measure, in the weak* topology (or weakly, in the language of probability theory) as $n \rightarrow \infty$. For the history of the subject and a list of references we refer the reader to the books \cite{BR, Fa}; % We cannot give an exhaustive overview here due to space limitations, and thus omit many important contributions, but
we shall discuss some recent results shortly.

It is natural to ask how fast the empirical measures converge, and in this work, we provide estimates on the expected rate of
convergence of associated quantities that measure the distance between counting measures and the uniform measure on the unit circle $\TT=\{e^{i\theta}\colon \theta \in [0,2\pi)\}$.

\subsection{Discrepancy of zeros and norms on polynomials}
For a polynomial $P_n$, we write $N(\alpha,\beta)$ for the number of elements of $\mathcal{Z}(P_n)$ that are contained
in the sector
\[S(\alpha,\beta)=\{z \in \CC: 0<\alpha\leq \arg z< \beta \leq 2\pi\};\]
later on we shall also work with annular sectors of the form
\begin{equation}
A_r(\alpha,\beta)=\{z \in \CC: r<|z|<1/r,\ \alpha\leq \arg z< \beta\},\quad  0<r<1.\label{angsector}
\end{equation}

A very classical result concerning the {\it angular discrepancy of zeros} is the theorem of Erd\H{o}s and Tur\'an \cite{ET}, which in its improved form due to Ganelius (see \cite {Gan}) asserts that
\begin{equation}
\left| \frac{N(\alpha, \beta)}{n}-\frac{\beta-\alpha}{2\pi}\right|\leq
\sqrt{\frac{2\pi}{\bf k}}\sqrt{\frac{1}{n}\log\left[\frac{\|P_n\|_{\infty}}{\sqrt{|c_0c_n|}}\right]};
\label{ETtheorem}
\end{equation}
here ${\bf k}=\sum_{k=0}^{\infty}(-1)^k/(2k+1)^2$ denotes Catalan's constant, and $\|P_n\|_{\infty}=\sup_{\mathbb{T}}|P_n|$. Suppose now that the coefficients of $P_n$ (random or deterministic) are uniformly
bounded with $|c_k|\leq K$, say. We then have the easy estimate
\[\log \|P_n\|_{\infty}\leq \log\left(\max_{z\in\TT} \sum_{k=0}^n|c_kz^k|\right)\leq \log(n+1)+\log K,\]
and by \eqref{ETtheorem}, it follows that
\[\left| \frac{N(\alpha, \beta)}{n}-\frac{\beta-\alpha}{2\pi}\right|\leq C\sqrt{\frac{\log(n+1)+\log K-\log\sqrt{|c_0c_n|}}{n}}.\]
Thus we see that if $c_0c_n \neq 0$ (almost surely), the discrepancy is of the order $\sqrt{\log n/n}$. Moreover,
it can be shown that this rate of decay is essentially best possible in the sense that one can construct (deterministic)
families of polynomials that exhibit discrepancy of this order (see \cite{ET, AM, P_Ark}). For general
coefficients however, it is hard to obtain effective estimates on $\|P_n\|_{\infty}$.

In order to state our results, we need to introduce certain additional norms on the circle. For a polynomial
$P_n$ and $0<p<\infty$, we set
\begin{equation*}
\|P_n\|_p=\left(\frac{1}{2\pi}\int_{\TT}|P_n(e^{i\theta})|^p\,d\theta\right)^{1/p}.%\quad \|P_n\|_{\infty}=\dis \sup_{\TT} |P_n|.
\end{equation*}
When $P_n(z)=\sum_{k=0}^n c_k z^k$ and $p=2$, we have
$\|P_n\|_2^2=\sum_{k=0}^n|c_k|^2$.

As usual, we define the Mahler measure (geometric mean) of $P_n$ by
\[M(P_n)=\exp\left(\frac{1}{2\pi}\int_0^{2\pi}\log|P_n(e^{i\theta})|d\theta\right),\]
and we set $m(P_n)=\log(M(P_n)).$ We also use \[M^+(P_n)=\exp\left(\frac{1}{2\pi}\int_0^{2\pi}\log^+|P_n(e^{i\theta})|d\theta\right)\]
and $m^+(P_n)=\log(M^+(P_n)).$ It is well known that
\begin{align} \label{1.3}
M(P_n) \le \|P_n\|_p \le \|P_n\|_q \le \|P_n\|_{\infty}, \quad 0<p<q<\infty.
\end{align}
Furthermore, the definitions immediately give that
\begin{align} \label{1.4}
M(P_n) \le M^+(P_n) \le \|P_n\|_{\infty}.
\end{align}

Recent papers on random polynomials and the behavior of their roots include \cite{HS, HN,IZ,IZa,KZ}.
Improving an earlier result of \v{S}paro and \v{S}ur, Ibragimov and Zaporozhets \cite{IZa} prove that the condition $\EE[\log^+|C_0|]<\infty$ is
both necessary and sufficient for almost sure asymptotic concentration of roots on the unit circumference. The paper \cite{KZ} deals with the interesting case of heavy-tailed coefficients, where $\EE[\log^+|C_0|]=\infty$ and the asymptotic distribution of roots is uniform in argument, but the radial positions of the roots accumulate on more than one circle. In \cite{HN}, Hughes and Nikeghbali deal with polynomials with not necessarily independent coefficients, and use estimates on $\EE[\log(\sum_k|C_k|)]$ to deduce, via Erd\H{o}s and Tur\'an's theorem, that roots concentrate on the unit circumference.

\subsection{Overview of the paper}
In this paper, we initiate a systematic quantitative study of convergence results for general random polynomials with coefficients that are not necessarily bounded, but satisfy the condition for asymptotic concentration. Using rather elementary methods, similar in spirit to those in \cite{HN}, we obtain results that are optimal unless further restrictions are introduced. If it is not stated otherwise, we shall impose the following standing assumption on the coefficients:
\begin{itemize}
\item $C_0,C_1,\ldots$ are independent and identically distributed (iid) complex random variables, with absolutely continuous distribution and $\EE[|C_0|^t]=\mu<\infty$ for some $t>0$.
\end{itemize}
We sometimes relax this assumption at the price of other more restrictive hypotheses.

Our main object of study are {\it expected discrepancies}, and we proceed as follows. We seek to control the discrepancy of zeros of a given random polynomial in annular sectors, first in terms of $m^+(P_n)$ and $m(P_n)$, and then $\log(\|P_n\|_2)$. To achieve this, we need to extend certain results of Mignotte and others; this is done in Section \ref{Sect2}. When estimating the expected discrepancy of zeros, we find it convenient to consider the quantity $\EE[\log(\max_k|C_k|)]$, since it appears in both upper and lower estimates on $\EE[\log(\|P_n\|_2)]$ and is amenable to elementary estimates. Under our standing assumption, we find that $\EE[\log(\max_k|C_k|)]=O(\log n)$. More precise statements are given in Section \ref{Sect3}. Our main result on the expected discrepancy is contained in Theorem \ref{MainThm}, and it deals with radial and angular parts of zeros simultaneously:
\begin{equation*}
\EE\left[\left| \tau_n\left(A_r(\alpha,\beta)\right) - \frac{\beta-\alpha}{2\pi}\right|\right]\\ \leq
C(r,t)\sqrt{\frac{\log(n+1) + \log \mu}{n}}
\end{equation*}
as $n\rightarrow \infty$; here, $C(r,t)$ is a constant that only depends on $r<1$ and on $t>0.$ Inspired by the work of Borwein, Erd\'elyi, and Littmann \cite{BEL}, we also derive a number of corollaries concerning
the {\it expected number of zeros} of random polynomials in polygons inscribed in the unit disk, and other natural sets. Under the additional assumption of a finite second moment, we indicate how our results extend to random coefficients that are not necessarily independent or identically distributed, yielding the same rate of decay for the expected discrepancy. Finally, in Section \ref{Sect4}, we present some elementary examples that illustrate that while $\EE[\log(\max_k|C_k|)]$ can be smaller in
special cases, $O(\log n)$ is the correct order of magnitude if no additional assumptions are made.

Full proofs are given in Section \ref{Sect5}.

\section{Annular discrepancies and norms on $\TT$}\label{Sect2}
The theorem of Erd\H{o}s and Tur\'an has been extended by several authors (see, for instance \cite{AB, AM, M, P, P_Ark}). In particular, Mignotte showed that one can use the weaker norm $M^+$ in estimates like \eqref{ETtheorem}, see \cite{M} and \cite{AM}. It is also desirable to include information about the radial behavior of the roots. This can be easily done by using Jensen's formula,
which leads to the following discrepancy estimate for the annular sectors
\eqref{angsector}.

\begin{prop}\label{mignotteprop}
Let $P_n(z)=\sum_{k=0}^n c_k z^k,\ c_k\in\CC$, and assume $c_0c_n\neq 0.$ For any $r\in(0,1)$ and $0\le \alpha < \beta < 2\pi,$ we have
\begin{align} \label{mignottediscr}
\left| \tau_n\left(A_r(\alpha,\beta)\right) - \frac{\beta-\alpha}{2\pi}\right| &\leq
\sqrt{\frac{2\pi}{{\bf k}}} \sqrt{\frac{1}{n}\,m^+\left(\frac{P_n}{\sqrt{|c_0c_n|}}\right)} \\ \nonumber &+ \frac{2}{n(1-r)} \, m\left(\frac{P_n}{\sqrt{|c_0c_n|}}\right).
\end{align}
\end{prop}
This estimate shows how close the zero counting measure $\tau_n$ is to $\mu_{\TT}.$ It is often more convenient to use the standard $L^p$ norms, especially the $L^2$ norm. We mention the following elementary but useful estimate.
\begin{prop} \label{Lpprop}
If $p\in(0,\infty)$ and $\|P_n\|_p \ge 1$, then
\[m^+(P_n) \le \log \|P_n\|_p + 1/(ep).\]
\end{prop}

\section{Expected discrepancies for random polynomials}\label{Sect3}
\subsection{Expectation of norms and the maximum of coefficients}
In order to estimate the {\it expected discrepancy} of zeros for classes of random polynomials, we apply Propositions \ref{mignotteprop} and \ref{Lpprop}, which requires an estimate on
\[\EE[\log \|P_n\|_2]=\frac{1}{2}\EE\left[\log\left(\sum_{k=0}^n |C_k|^2\right)\right].\]
Expected values of $L^p$ norms and Mahler measures for random polynomials have been considered by a number of authors. For instance,
Fielding (see \cite{F}) has computed $\EE[\log M(P_n)]$ for polynomials with coefficents uniformly distributed on $\TT$; he obtains
\[\EE[\log M(P_n)]=\frac{1}{2}\log(n+1)-\frac{\mathbf{\gamma}}{2}+O(n^{-1/2+\delta})\]
for arbitrary $\delta>0$. Here, $\mathbf{\gamma}$ denotes Euler's constant. Expected values of $L^p$ norms of random polynomials have
also been studied by Borwein and Lockhart, and Choi and Mossinghoff (see \cite{BL, CM}), among others.

It is elementary that
\[\max_{k=0,\ldots, n}|C_k| \leq \left( \sum_{k=0}^n|C_k|^p\right)^{1/p}\leq (n+1)^{1/p}\max_{k=0,\ldots, n}|C_k|,\]
and from this it follows that
\begin{equation}
\EE[\log(\max_k|C_k|)]\leq \EE[\log \|P_n\|_2]\leq \frac{1}{2}\log(n+1)+\EE[\log(\max_k |C_k|)].
\label{comparison}
\end{equation}
Hence, we obtain an upper estimate for the discrepancy in terms of
\[\frac{1}{2}\log (n+1)+\EE\left[\log \max_{k=0,\ldots,n}|C_k|\right]-\EE[\log|C_0|],\]
and thus the same order of decay as for bounded coefficients if we can show, for instance, that $\EE\left[\log(\max_{k=0,\ldots,n}|C_k|)\right] = O(\log n)$. On the other hand, the lower bound in \eqref{comparison} means that approaching the expected discrepancy via $\ell^q$ norms of coefficients (which bound $\|P_n\|_p$ from below for $1\leq p \leq 2$) will not work if the expected value is too large; in view of Ibragimov and Zaporozhets' result \cite{IZa}, the roots do not exhibit (almost sure) asymptotic clustering if $\EE[\log^+|C_0|]=\infty$.

\subsection{Expectation of $\log(\max_k|C_k|)$}
We set $R_{C}(r)=\PP(|C_0|\leq r)$ and let $\rho_{C}(r)=R'_{C}(r)$, $r \in [0,\infty)$, denote the density of the non-negative random variable $|C_0|$, which exists and satisfies $\int_0^{\infty} r^t \rho_C(r)dr<\infty$ by our standing assumptions. Indeed, if $\rho(r,\theta)$ is the density of $C_0$ with respect to the area measure, then we can write $\PP(|C_0|\le r)=\int_0^r \left(\int_0^{2\pi} s\,\rho(s,\theta) \,d\theta\right) ds=\int_0^r \rho_C(s)\, ds.$ As is standard in order statistics (see \cite{DaNa}), we now express the density of the random variable $Y_n=\max_{k=0,\ldots, n}|C_k|$ in terms of $\rho_{C}$.

\begin{lem} \label{norderlemma}
Suppose $n\geq 1$. Then the density of $Y_n=\max_{k=0,\ldots,n}|C_k|$ is given by
\begin{equation} \label{3.2}
\rho_{Y_n}(r)=(n+1)\rho_C(r)[R_C(r)]^n.
\end{equation}
\end{lem}
Using Lemma \ref{norderlemma}, we proceed by estimating
\begin{equation*}
\EE\left[\log Y_n\right]=\int_0^{\infty} (n+1)\rho_C(r)[R_C(r)]^n  \log{r}\, dr,
\label{generallogexp}
\end{equation*}
or equivalently, with $x(u)=R_C^{-1}(u)$,
\begin{equation*}
\EE\left[\log Y_n\right]=\int_0^{1} (n+1) \log(x(u))u^n\,du.
\end{equation*}
We recall that $\EE[|C_0|^t]=\mu<\infty$, and Jensen's inequality now yields
\begin{lem}\label{Elogest}
We have
\begin{equation}
\EE\left[\log Y_n\right]\leq \frac{1}{t}(\log(n+1)+\log\mu).
\label{generallogest}
\end{equation}
\end{lem}

\subsection{Main Results}
Combining Propositions \ref{mignotteprop} and \ref{Lpprop} with \eqref{comparison}, and using Lemma \ref{Elogest}, we obtain the desired expected discrepancy result.
\begin{thm}\label{MainThm}
If the coefficients of $P_n(z)=\sum_{k=0}^{n}C_k z^k$ are iid complex random variables with absolutely continuous distribution and $\EE[|C_0|^t]<\infty$, then we have for all large $n\in\NN$ that
\begin{multline} \label{3.4}
\EE\left[\left| \tau_n\left(A_r(\alpha,\beta)\right) - \frac{\beta-\alpha}{2\pi}\right|\right]\\ \leq
\left(\sqrt{\frac{2\pi}{{\bf k}}} + \frac{2}{1-r} \right)
\sqrt{\frac{\frac{t+2}{2t}\log(n+1) + \frac{1}{t}\log \EE[|C_0|^t]+\frac{1}{2e}-\EE[\log|C_0|]}{n}}.
\end{multline}
\end{thm}

%\begin{rem}
%In some cases, the application of Jensen's inequality is wasteful and $\EE[\log Y_n]$ can be shown to grow much slower than $\log n$. On the other hand, if the distribution of $C_0$ has heavy tails, we do have logarithmic growth of this expected value. We give examples to illustrate these facts in the next section.
%\end{rem}

Note that for any set $E\subset\CC$, the number of zeros of the polynomial $P_n$ in $E$ is given by $n\tau_n(E).$ Thus our results may be stated in terms of the expected number of zeros in certain sets (see \cite[Ch. 5]{Adler} for a general theorem). We give several examples of such statements below. We point out
that for some special families of coefficients, explicit formulas for the density of zeros in a given set are available, see \cite{BR}, \cite{Fa}, and the references therein.
The following result states that the expected number of zeros in any compact set that does not meet $\TT$ is of the order $O(\log{n})$ as $n\to\infty.$

\begin{prop}\label{compactprop}
Let $E\subset\CC$ be a compact set such that $E\cap\TT=\emptyset,$ and set $d:=\textup{dist}(E,\TT).$ The expected number of zeros of $P_n$ in $E$ satisfies
\begin{align} \label{3.5}
\EE\left[n \tau_n(E)\right] \leq
\frac{d+1}{d}
\left(\frac{t+2}{t}\log(n+1) + \frac{2}{t}\log \EE[|C_0|^t] - 2\EE[\log|C_0|]\right).
\end{align}
\end{prop}
If the set $E$ does not have a ``close'' contact with the unit circle $\TT$, then the expected number of zeros in $E$ still remains of the order $o(n)$ as $n\to\infty.$ In particular, we have the following

\begin{prop}\label{polygonprop}
If $E$ is a polygon inscribed in $\TT,$ then the expected number of zeros of $P_n$ in $E$ satisfies
\begin{align} \label{3.6}
\EE\left[n \tau_n(E)\right] = O( \sqrt{n\,\log{n}})\quad \mbox{as } n\to\infty.
\end{align}
\end{prop}

An estimate of this type for deterministic polynomials with restricted coefficients was proved in \cite{BE}. Another estimate for the zeros of deterministic polynomials in the disks $D_r(w)=\{z\in\CC:|z-w|<r\},\ w\in\TT,$ is contained in \cite{BEL}. We provide an analogue for the random polynomials below.

\begin{prop}\label{diskprop}
For  $D_r(w)$ with $\ w\in\TT,\ r<2,$ the expected number of zeros of $P_n$ in $D_r(w)$ satisfies
\begin{align} \label{3.7}
\EE\left[n \tau_n(D_r(w))\right] =  \frac{2 \arcsin(r/2)}{\pi}\, n + O\left(\sqrt{n\log{n}}\right)\quad \mbox{as } n\to\infty.
\end{align}
\end{prop}

\subsection{Remarks on the non-iid case}
Under additional assumptions on $\rho_C$, we can say more. For instance, it is known (see \cite[Chapter 4]{DaNa}) that if $\textrm{Var}(|C_0|)=\sigma^2<\infty$, then
\[\EE[Y_n]\leq \mu+\sigma\frac{n}{\sqrt{2n+1}},\]
and that equality can be achieved for each $n$ with a particular choice of distribution. In
this case, \eqref{generallogest} takes on the asymptotic form
$\EE[\log Y_n^C]\leq (1/2)\log n+O(1)$.

Moreover, the requirement that $C_0, C_1, \ldots$ be independent and identically distributed can be dropped if the second moments of their moduli are finite.
Namely, if
\begin{equation}
\EE[|C_0|]=\EE[|C_1|]=\cdots=\mu \quad \textrm{and}\quad
\textrm{Var}(|C_0|)=\textrm{Var}(|C_1|)=\cdots=\sigma^2,\label{noniidconditions}
\end{equation}
then it follows from a result of Arnold and Groeneveld (see
\cite[Chapter 5]{DaNa}) that
\begin{equation}
\EE[Y_n]\leq \mu+\sigma \sqrt{n}.\label{Enoniid}
\end{equation}
Returning once more to \eqref{comparison}, and applying Jensen's inequality, we
deduce the following version of our result on expected discrepancies (which can then be applied to extend the other results in the previous
subsection).
\begin{thm}\label{MainThmB}
 If the (not necessarily iid) coefficients of $P_n(z)=\sum_{k=0}^{n}C_kz^k$ have absolutely continuous distributions, and satisfy \eqref{noniidconditions}, then
\begin{multline*}
\EE\left[\left| \tau_n\left(A_r(\alpha,\beta)\right) - \frac{\beta-\alpha}{2\pi}\right|\right]\\ \leq
\left(\sqrt{\frac{2\pi}{{\bf k}}} + \frac{2}{1-r} \right)
\sqrt{\frac{\log(n+1)-\frac{1}{2}\EE[\log|C_0|]-\frac{1}{2}\EE[\log |C_n|]+O(1)}{n}},
\end{multline*}
as $n \to \infty$.
\end{thm}

\section{Examples}\label{Sect4}
\subsection{Gaussian coefficients}
Let $C_0, C_1,\ldots$ be iid with density $\rho^G$ given by
\begin{equation}
d\mu^G(z)=\rho^G(z)dA(z)=e^{- r^2}\frac{r dr d\theta}{\pi};
\label{gaussdens}
\end{equation}
that is, the coefficients are simply centered two-dimensional Gaussians. We readily
compute that $\EE[\log |C_0|]=-\gamma/2$.

We determine the asymptotics of $\EE[\log(Y_n^G)]$ by elementary computations.
\begin{lem}\label{gausslemma}
Let $n\geq 1$. Then
\begin{equation*}
\EE[\log Y^G_n]=2(n+1)\int_0^{\infty}x \log x \,e^{-x^2}(1-e^{-x^2})^ndx,
\label{gaussexpect}
\end{equation*}
and, asymptotically,
\begin{equation*}
\EE[\log Y^G_n]=-\frac{\gamma}{2}+\frac{1}{2}\sum_{k=2}^{n+1} (-1)^k\left(\begin{array}{c}n+1\\k\end{array}\right)\log k \leq \log\log n+O(1), \quad n\rightarrow \infty.
\label{gaussasymptotics}
\end{equation*}
\end{lem}
We obtain a slightly better expected discrepancy result (cf. \cite[Sect. 3]{Arn}).
\begin{prop}
For large enough $n$,
\begin{equation*}
\EE\left[\left| \tau_n\left(A_r(\alpha,\beta)\right) - \frac{\beta-\alpha}{2\pi}\right|\right]\leq
\left(\sqrt{\frac{2\pi}{{\bf k}}} +\frac{2}{1-r} \right)
\sqrt{\frac{\log n+\log\log n+O(1)}{n}}.
\end{equation*}
\end{prop}
% \begin{rem} Essentially the same result follows from the following general theorem from order statistics (see \cite{DaNa}): if the function
% $R$ is convex, then $\EE[Y_n]\leq R^{-1}(n/(n+1))$.
% We have $R^{-1}(u)=x(u)=(-\log(1-u))^{1/2}$ in the present Gaussian case.
% \end{rem}
\subsection{Heavy tails}
We now turn to the case of heavy-tailed coefficients.

Let $\alpha>1$ and $C_0, C_1, \ldots$ be iid with a Pareto-type density
\begin{equation}
\rho^{P_{\alpha}}(z)=\left\{\begin{array}{cc}\frac{\alpha-1}{2r^{\alpha+1}}, & |z|>1,\\
0, & \textrm{otherwise}.\end{array}\right.
\label{paretodens}
\end{equation}
Note that $\EE[|C_0|]=(\alpha-1)/(\alpha-2)$ for $\alpha>2$ and is infinite otherwise, and that $\EE[\log|C_0|]=\EE[\log^+|C_0|]<\infty$ for every $\alpha>0$.

In view of \eqref{generallogexp}, we obtain the following.
\begin{lem}\label{heavylemma}
Let $n\geq 1$. Then
\begin{equation*}
\EE[\log(Y_n^{P_{\alpha}})]=(\alpha-1)(n+1)\int_1^{\infty}\frac{\log x}{x^{\alpha}}
(1-x^{-(\alpha-1)})^ndx=\frac{1}{\alpha-1}\mathbf{H}_{n+1},
\end{equation*}
where $\mathbf{H}_n=\sum_{k=1}^n 1/k$ is the $n$th harmonic number.
\end{lem}
It is well-known that $\mathbf{H}_n=\log n+O(1)$, and so the
estimate \eqref{generallogest} cannot be improved in general.

\section{Proofs}\label{Sect5}

Let $D_r=\{z\in\CC: |z|<r\},\ r>0$. We need the following consequence of Jensen's formula.
\begin{lem} \label{lemma5.1}
If $P_n(z)=\sum_{k=0}^n c_k z^k,\ c_k\in\CC$ and $c_0c_n\neq 0,$ then for any $r\in(0,1)$ we have
\begin{equation} \label{5.1}
\tau_n\left(\overline{D_r}\right) \leq \frac{m(P_n) - \log|c_0|}{n(1-r)}
\end{equation}
and
\begin{equation} \label{5.2}
\tau_n\left(\CC\setminus D_{1/r}\right) \leq \frac{m(P_n) - \log|c_n|}{n(1-r)}.
\end{equation}
\end{lem}

\begin{proof}
Let $P_n(z)= c_n \prod_{k=1}^n (z-z_k).$ Using Jensen's formula, we obtain that
\begin{align*}
m(P_n)-\log|c_0| &= \frac{1}{2\pi}\int_0^{2\pi}\log|P_n(e^{i\theta})|d\theta -\log|c_0| = \sum_{|z_k|<1} \log\frac{1}{|z_k|} \\ &\ge \sum_{|z_k|\le r} \log\frac{1}{|z_k|} \ge n \tau_n\left(\overline{D_r}\right) \log\frac{1}{r} \ge n \tau_n\left(\overline{D_r}\right) (1-r).
\end{align*}
Thus the first estimate follows, and we can apply it to the reciprocal polynomial $P_n^*(z)=z^n \overline{P_n(1/\bar z)}.$ Note that the zeros of $P_n^*$ are $1/\bar z_k,\ k=1,\ldots,n,$ and its constant term is $\bar c_n.$ Since  $|P_n^*(z)|=|P_n(z)|$ for $|z|=1,$ we have that $m(P_n^*)=m(P_n)$. Hence \eqref{5.1} applied to $P_n^*$ now gives that
\[\tau_n\left(\CC\setminus D_{1/r}\right) \leq \frac{m(P_n^*) - \log|\bar c_n|}{n(1-r)} = \frac{m(P_n) - \log|c_n|}{n(1-r)}.\]
\end{proof}

\begin{proof}[Proof of Proposition \ref{mignotteprop}]
Under the assumptions of this proposition, the result of Mignotte \cite{M} (see also \cite{AM}) gives
\begin{equation*}
\left| \tau_n\left(S(\alpha,\beta)\right) - \frac{\beta-\alpha}{2\pi}\right|\leq
\sqrt{\frac{2\pi}{{\bf k}}}\, \sqrt{\frac{1}{n}\,m^+\left(\frac{P_n}{\sqrt{|c_0c_n|}}\right)}.
\end{equation*}
On the other hand, applying Lemma \ref{lemma5.1}, we obtain that
\begin{align} \label{5.3}
\tau_n\left(\CC\setminus A_r(\alpha,\beta)\right) &\leq \frac{2}{n(1-r)} \, m\left(\frac{P_n}{\sqrt{|c_0c_n|}}\right).
\end{align}
Since $\tau_n\left(A_r(\alpha,\beta)\right) =  \tau_n\left(S(\alpha,\beta)\right) - \tau_n\left(\CC\setminus A_r(\alpha,\beta)\right),$ \eqref{mignottediscr} follows as a combination of the above estimates.
\end{proof}

\begin{proof}[Proof of Proposition \ref{Lpprop}]
Consider the set $E:=\{\theta\in[0,2\pi): |P_n(e^{i\theta})| \ge 1\},$ and denote its length by $|E|.$ Our assumption $\|P_n\|_p \ge 1$ implies that $|E|\neq 0$. We use concavity of $\log$ and Jensen's inequality in the following estimate
\begin{align*}
m^+(P_n) &= \frac{1}{2\pi}\int_0^{2\pi}\log^+|P_n(e^{i\theta})|\,d\theta = \frac{1}{2\pi p}\int_E \log|P_n(e^{i\theta})|^p\,d\theta \\ &= \frac{|E|}{2\pi p}\int_E \log|P_n(e^{i\theta})|^p\,\frac{d\theta}{|E|} \le \frac{|E|}{2\pi p}\log \left(\int_E |P_n(e^{i\theta})|^p\,\frac{d\theta}{|E|}\right) \\ &= \frac{|E|}{2\pi p} \left( \log\frac{2\pi}{|E|} + \log\left(\frac{1}{2\pi}\int_E |P_n(e^{i\theta})|^p\,d\theta\right)\right) \\ &\le \frac{|E|}{2\pi p} \log\frac{2\pi}{|E|} + \frac{|E|}{2\pi p} \log\left(\frac{1}{2\pi}\int_0^{2\pi}  |P_n(e^{i\theta})|^p\,d\theta\right) \\ &\le \frac{1}{p}\sup_{x\in(0,1]} x\log \frac{1}{x} + \log\|P_n\|_p =  \frac{1}{ep} + \log\|P_n\|_p.
\end{align*}
\end{proof}

\begin{proof}[Proof of Lemmas \ref{norderlemma} and \ref{Elogest}]
In view of the fact that the $C_k$'s are independent, we have
\begin{align*}
F_{Y_n}(r)&=\PP(Y_n\leq r)=\PP(|C_0|\leq r, |C_1|\leq r, \ldots, |C_n|\leq r)\\
&=\PP(|C_0|\leq r)\PP(|C_1|\leq r)\cdots \PP(|C_n|\leq r),
\end{align*}
and since they are identically distributed,
$F_{Y_n}(r)=[\PP(|C_0|\leq r)]^{n+1}$.
The statement of Lemma \ref{norderlemma} now follows upon differentiation.

By Jensen's inequality,
\[\EE[\log Y_n]\leq \frac{1}{t}\log \EE[Y_n^t],\]
and using \eqref{3.2}, and $R_C(x)\leq 1$, we obtain
\begin{align*}
\EE[Y_n^t] = \int_0^{\infty} x^t (n+1)\rho_C(x)[R_C(x)]^n\,dx &\leq (n+1) \int_0^{\infty}x^t\rho_C(x)dx \\ &= (n+1)\EE[|C_0|^t],
\end{align*}
%We change variables by setting $u=R_C(x)$; then $du=R'_C(x)dx$ and $x=x(u)=R^{-1}_C(u)$. Moreover,
%since $R_C$ is a distribution on $[0, \infty)$, we obtain
%\[\EE[Y_n]=n\int_0^1x(u)^t u^{n-1}du\leq n\int_0^1x(u)^tdu.\]
%The latter integral is equal to $\EE[|C_0|^t]$ by definition.
\end{proof}

\begin{proof}[Proofs of Theorems \ref{MainThm} and \ref{MainThmB}]
We first apply Proposition \ref{mignotteprop} and Jensen's inequality to obtain
\begin{align*}
\EE\left[\left| \tau_n\left(A_r(\alpha,\beta)\right) - \frac{\beta-\alpha}{2\pi}\right|\right] &\leq
\sqrt{\frac{2\pi}{{\bf k}}} \sqrt{\frac{1}{n}\, \EE\left[ m^+\left(\frac{P_n}{\sqrt{|C_0 C_n|}}\right) \right]} \\  &+ \frac{2}{n(1-r)} \, \EE\left[ m\left(\frac{P_n}{\sqrt{|C_0C_n|}}\right)\right].
\end{align*}
A combination of Proposition \ref{Lpprop} with \eqref{comparison} and Lemma \ref{Elogest} gives that \begin{align} \label{5.4}
\EE\left[ m^+\left(\frac{P_n}{\sqrt{|C_0 C_n|}}\right) \right] &\le \EE\left[ \log \left\|\frac{P_n}{\sqrt{|C_0 C_n|}}\right\|_2 \right] + \frac{1}{2e} \nonumber \\ &\le \frac{\log(n+1)}{2} + \EE\left[\log \max_{k=0,\ldots,n}|C_k|\right]-\EE[\log|C_0|] + \frac{1}{2e} \nonumber \\ &\le \frac{t+2}{2t}\log(n+1) + \frac{1}{t}\log \EE[|C_0|^t]-\EE[\log|C_0|]+\frac{1}{2e}.
\end{align}
To justify the use of Proposition \ref{Lpprop}, we note that $ \left\|P_n/\sqrt{|C_0 C_n|}\right\|_2 \ge 1,$ which is a consequence of the fact that $\|P_n\|_2 \ge |C_k|,\ k=0,\ldots,n.$

Using \eqref{1.3}, \eqref{comparison} and Lemma \ref{Elogest}, we obtain that
\begin{align*}
\EE\left[ m\left(\frac{P_n}{\sqrt{|C_0 C_n|}}\right) \right] &\le \EE\left[ \log \left\|\frac{P_n}{\sqrt{|C_0 C_n|}}\right\|_2 \right] \\ &\le \frac{\log(n+1)}{2} + \EE\left[\log \max_{k=0,\ldots,n}|C_k|\right]-\EE[\log|C_0|] \\ &\le \frac{t+2}{2t}\log(n+1) + \frac{1}{t}\log \EE[|C_0|^t]-\EE[\log|C_0|].
\end{align*}
Thus \eqref{3.4} follows from the above estimates.

The proof of Theorem \ref{MainThmB} is similar. We first argue as above,
bounding the discrepancy in terms of $\EE\left[\log \left\|P_n/\sqrt{|C_0 C_n|}\right\|_2\right]$. We then have
\begin{align*}
\EE\left[\log \left\|P_n/\sqrt{|C_0 C_n|}\right\|_2\right]
&\leq \frac{\log(n+1)}{2}+\EE[\log(\max_k|C_k|)]\\ &-\frac{\EE[\log|C_0|]+\EE[\log|C_n|]}{2}.
\end{align*}
We now appeal to \eqref{Enoniid} and Jensen's inequality instead of Lemma \ref{Elogest} to obtain
\[\EE[\log (\max_k|C_k)]\leq \log[\mu +\sigma \sqrt{n}]\leq \frac{1}{2}\log(n+1)+O(1),\]and the theorem follows.
\end{proof}

\begin{proof}[Proof of Proposition \ref{compactprop}]
Using Lemma \ref{lemma5.1}, we obtain as in \eqref{5.3} that
\begin{align*}
\tau_n\left(\CC\setminus A_r(0,2\pi)\right) &\leq \frac{2}{n(1-r)} \, m\left(\frac{P_n}{\sqrt{|C_0C_n|}}\right).
\end{align*}
If $r$ is selected so that $E\subset\CC\setminus A_r(0,2\pi)$, then
\begin{align*}
\EE\left[n \tau_n(E)\right] &\leq \frac{2}{1-r}\, \EE\left[ m\left(\frac{P_n}{\sqrt{|C_0 C_n|}}\right) \right] \le \frac{2}{1-r}\, \EE\left[ \log \left\|\frac{P_n}{\sqrt{|C_0 C_n|}}\right\|_2 \right] \\ &\le
 \frac{2}{1-r}\,\left(\frac{t+2}{2t}\log(n+1) + \frac{1}{t} \log \EE[|C_0|^t] - \EE[\log|C_0|]\right),
\end{align*}
where we used \eqref{1.3}, \eqref{comparison} and Lemma \ref{Elogest}. An elementary argument shows that $r=1/(\textup{dist}(E,\TT)+1)$ implies $E\subset\CC\setminus A_r(0,2\pi)$.
\end{proof}

\begin{proof}[Proof of Proposition \ref{polygonprop}]
Suppose that the vertices of our polygon are located at the points $e^{i\theta_j},\ j=1,\ldots,k.$ It follows from a simple Euclidean geometry consideration that the polygon is contained in the union of the
closed disk $U=\{z\in\CC:|z| \le 1-\sqrt{\log{n}/n}\}$ and sectors $S_j=\{z\in\CC: |\arg z -\theta_j|<C \sqrt{\log{n}/n}\},$ for each $n\in\NN,$ where  $C>0$ is a constant that depends only on the polygon. Note that Proposition \ref{compactprop} applies to $U$, so that \eqref{3.5} gives
\begin{align} \label{5.5}
\EE\left[n \tau_n(U)\right] = O\left(\sqrt{n\log{n}}\right)\quad\mbox{as } n\to\infty.
\end{align}
We recall the result of Mignotte \cite{M} (see also \cite{AM})
\begin{equation*}
\left| \tau_n\left(S(\alpha,\beta)\right) - \frac{\beta-\alpha}{2\pi}\right|\leq
\sqrt{\frac{2\pi}{{\bf k}}}\, \sqrt{\frac{1}{n}\,m^+\left(\frac{P_n}{\sqrt{|C_0C_n|}}\right)},
\end{equation*}
and apply it to each sector $S_j.$ It follows from the above estimate and Jensen's inequality that
\begin{align*}
\EE\left[n \tau_n(S_j)\right] = O\left(\sqrt{n\log{n}}\right) + O\left(\sqrt{n\EE\left[ m^+\left(\frac{P_n}{\sqrt{|C_0 C_n|}}\right) \right]}\right)\quad\mbox{as } n\to\infty,
\end{align*}
for each $j=1,\ldots,k.$ Using the already available inequality \eqref{5.4}, we conclude that the second term is of the same order as the first, so that
\begin{align*}
\EE\left[n \tau_n(S_j)\right] = O\left(\sqrt{n\log{n}}\right)\quad\mbox{as } n\to\infty,
\end{align*}
for each $j=1,\ldots,k.$ Thus \eqref{3.6} is a consequence of \eqref{5.5} and the above equation.
\end{proof}

\begin{proof}[Proof of Proposition \ref{diskprop}]
We follow ideas similar to those used in the proof of Proposition  \ref{polygonprop}. Note that the intersection of the disk $D_r(w)$ and $\TT$ is an arc with endpoints $e^{i\alpha}$ and $e^{i\beta}$, where $\alpha<\beta<\alpha+2\pi$ and $\beta-\alpha=4\arcsin(r/2).$ Furthermore, $D_r(w)$ is contained in the union of the circular sector $S=\{z\in\CC: \alpha - C\sqrt{\log{n}/n} < \arg z < \beta + C\sqrt{\log{n}/n}\},$ where $C>0$ depends only on $r$, and the set $F=\{z\in\CC:|z| \le 1-\sqrt{\log{n}/n} \mbox{  or  } |z| \ge 1+\sqrt{\log{n}/n}\}$.  Proposition \ref{compactprop} implies that
\begin{align*}
\EE\left[n \tau_n(F)\right] = O\left(\sqrt{n\log{n}}\right)\quad\mbox{as } n\to\infty.
\end{align*}
On the other hand, Mignotte's estimate gives that
\begin{equation*}
\EE\left[n\tau_n\left(S\right) - \frac{\beta-\alpha}{2\pi}n\right] = O\left(\sqrt{n\log{n}}\right) \quad\mbox{as } n\to\infty,
\end{equation*}
arguing as in the proof of Proposition  \ref{polygonprop}. Combining the last two equations, we arrive at \eqref{3.7}.
\end{proof}

%\begin{itemize}
%\item The Jensen argument here yields the same conclusion as in \cite[Prop. 6]{HN}: for iid variables, their argument yields $(1/s)\log\left(\sum_{k=0}^N\EE[|a_k|^s]\right)=(1/s)\log(N\EE[|a_0|^s])$ in the range $0<s\leq 1$. In general, they do not assume independence.
%\end{itemize}
\begin{proof}[Proofs of Lemmas \ref{gausslemma} and \ref{heavylemma}]
The first statement follows from Lemma \ref{norderlemma} and the fact that $R^G(x)=\int_0^xr\,e^{-r^2}dr=(1-e^{-x^2})/2$; the series expression is readily obtained by using the binomial theorem.

To analyze the asymptotics of $\EE[\log Y^G_n]$, we first perform the obvious change of variables $u=\exp(-x^2)$ to obtain
\begin{align*}
\EE[\log Y^G_n] &= \frac{n+1}{2}\int_0^1\log\log\left(\frac{1}{u}\right)(1-u)^n\,du \\ &=I_1(n)+I_2(n)=\int_0^{\frac{1}{(n+1)^2}}+\int_{\frac{1}{(n+1)^2}}^{1}.
\end{align*}
Since $\log\log(1/u)$ is decreasing, for all sufficiently large $n$, the latter integral admits the estimate
\begin{align*}
I_2(n) &\leq \frac{(n+1)\log\log (n+1)^2}{2}\int_{\frac{1}{(n+1)^2}}^1(1-u)^n du \\ &=\frac{\log\log (n+1)^2}{2}\left(1-\frac{1}{(n+1)^2}\right)^{n+1} \leq \frac{\log\log (n+1)^2}{2}.
\end{align*}
We next show that the contribution arising from the first integral is negligible. An integration by parts shows that
\[\int_0^x \log\log\left(\frac{1}{u}\right)du=x\log\log\left(\frac{1}{x}\right)-\mathrm{li}(x), \quad x>0,\]
where $\mathrm{li}(x)=\int_0^x1/\log(u)du$ is the logarithmic integral. We use this identity together with the crude estimate $(1-u)^n \leq 1$ to obtain
\begin{align} \label{integralatzero}
I_1(n) &\leq \frac{n+1}{2}\int_0^{\frac{1}{(n+1)^2}}\log\log\left(\frac{1}{u}\right)du \\ &=
\frac{1}{2(n+1)}\log\log (n+1)^2-\frac{n+1}{2}\textrm{li}\left(\frac{1}{(n+1)^2}\right). \nonumber
\end{align}
The first term on the right-hand side in \eqref{integralatzero} clearly tends to zero as $n\rightarrow \infty$. We then note that the function $x\mapsto \mathrm{li}(x^2)/x$ is continuous on $(0,1)$, and by L'H\^opital's rule,
$\lim_{x\rightarrow 0}\mathrm{li}(x^2)/x=\lim_{x\rightarrow 0}x/\log(x)=0$. Hence $I_1(n)=o(1)$ as required.

For $\rho^{P_{\alpha}}$, we have
\begin{align*}
\EE[\log(Y_n^{P_{\alpha}})] &=(n+1) \int_0^{1} \log\left[\frac{1}{(1-u)^{\frac{1}{\alpha-1}}}\right]u^n\,du \\ &= \frac{n+1}{\alpha-1}\int_0^{1}\log\frac{1}{1-u}\,u^n\,du,
\end{align*}
and the result follows from Euler's formula $\mathbf{H}_n=\int_0^1\frac{1-u^n}{1-u}du$
and an integration by parts.
\end{proof}
\section*{Acknowledgments}
The second author thanks Martin Bender for interesting discussions.

\end{document}